\documentclass[12pt]{article}

\usepackage{amsmath,amsthm,amssymb,mathrsfs}
\usepackage{graphicx,mdframed}

\font\tencmmib=cmmib10 \skewchar\tencmmib '60
\newfam\cmmibfam
\textfont\cmmibfam=\tencmmib

\def\lessim{\ \lower4pt\hbox{$
		\buildrel{\displaystyle <}\over\sim$}\ }
\def\gessim{\ \lower4pt\hbox{$\buildrel{\displaystyle >}
		\over\sim$}\ }

\newcommand{\e}{\mathbb{E}}
\newcommand{\p}{\mathbb{P}}

\newcommand{\Var}{\mathrm{Var}}

\newtheorem{lemma}{\bf Lemma}

\newtheorem{theorem}{\bf Theorem}

\newtheorem{remark}{\bf Remark}

\newtheorem{proposition}{\bf Proposition}

\newmdtheoremenv{theo}{Theorem}

\newenvironment{Proof of lemma}{\noindent{\bf Proof of Lemma}}{\hfill$\Box$\newline}
\newenvironment{Proof of theorem}{\noindent{\it Proof of Theorem}}{\hfill\scriptsize{$\Box$}\newline}
\newenvironment{Proof of theorems}{\noindent{\bf Proof of Theorems}}{\hfill$\Box$\newline}
\newenvironment{Proof of proposition}{\noindent{\bf Proof of Proposition}}{\hfill$\Box$\newline}
\newenvironment{Proof of propositions}{\noindent{\bf Proof of Propositions}}{\hfill$\Box$\newline}
\newenvironment{Proof of exercise}{\noindent{\it Proof of Exercise:}}{\hfill$\Box$}
\begin{document}

	\title{A Gaussian Convexity for Logarithmic Moment\\ 
		Generating Functions with Applications in Spin Glasses}	

 \author{Wei-Kuo Chen\thanks{School of Mathematics, University of Minnesota, Email: wkchen@umn.edu. Partly supported by NSF grants DMS-1752184 and DMS-2246715 and Simons Foundation grant 1027727}}

\maketitle

\begin{abstract}
    For any convex function $F$ of $n$-dimensional Gaussian  vector $g$ with $\e e^{\lambda F(g)}<\infty$ for any $\lambda>0$, we show that $\lambda^{-1}\ln \e e^{\lambda F(g)}
    $ is convex  in  $\lambda\in\mathbb{R}$. 
    Based on this convexity, we draw three major consequences. The first  recovers a version of the Paouris-Valettas lower deviation  inequality for Gaussian convex functions with an improved exponent.
    The second establishes a quantitative bound for the Dotsenko-Franz-M\'ezard conjecture arising from the study of the Sherrington-Kirkpatrick (SK) mean-field spin glass model, which states that in the absence of external field, the annealed free energy of negative replica is asymptotically equal to the free energy. The last further establishes the differentiability for this annealed free energy with respect to the negative replica variable at any temperature and external field.
\end{abstract}

\section{Main Results}

Convex functions of Gaussian vectors are prominent objects in a variety of mathematical subjects. For example, they played an essential role in the estimation of small ball probability for Gaussian processes in the context of convex geometry \cite{paouris2018gaussian}. Moreover, it is largely used in the formulation of the disordered free energies in a variety of statistical physics models such as spin glasses \cite{charbonneau2023spin,mezard1987spin} and  directed polymers \cite{comets2017directed}. 

The goal of this short note is to establish a new convexity for the logarithmic moment generating function for any convex functions of Gaussian vectors. Let $F:\mathbb{R}^n\to \mathbb{R}$ be a measurable function and $g$ be an $n$-dimensional standard Gaussian vector. The following is our main result.

\begin{theorem}\label{convexity}
	Assume that $F$ is convex on $\mathbb{R}^n$ with $\e e^{\lambda F(g)}<\infty$ for all $\lambda>0.$ Then the function $\Phi:\mathbb{R}\to \mathbb{R}$ defined as 
	\begin{align*}
		\Phi(\lambda)=\frac{1}{\lambda}\ln \e e^{\lambda F(g)}
	\end{align*}
	for all $\lambda \neq 0$ and $\Phi(0)=\e F(g)$ is convex on $\mathbb{R}$.
\end{theorem}

Note that $\Phi$ is finite everywhere due to the assumption $\e e^{\lambda F(g)}<\infty$ for all $\lambda>0$ and the fact that the convexity of $F$ ensures $\e e^{\lambda F(g)}<\infty$ for all $\lambda<0$ since $F$ is supported from below by a linear function. We mention that the convexity on $\lambda>0$ was observed earlier in many examples of $F$, see, e.g., \cite{auffinger2015parisi,guerra2018emergence}. The convexity of $\Phi$ for $\lambda<0$ is new and based on this, we establish some properties for $\Phi$ for negative $\lambda.$

\begin{theorem}\label{thm2}
	Assume that $F$ is convex with $\e F(g)^2<\infty.$  Then
	\begin{enumerate}
		\item For any $\lambda\leq 0,$ 
		\begin{align}\label{add:eq1}\ln \e e^{\lambda (F(g)-\e F(g))}\leq \frac{\lambda^2}{2}\Var(F(g)).
		\end{align}
		\item For any $\lambda\leq 0,$
		\begin{align}\label{thm2:eq1}
			\bigl|\Phi(\lambda)-\e F(g)\bigr|\leq \frac{|\lambda|}{2}\Var(F(g)).
		\end{align}
		\item For any $\lambda,\lambda'<0,$
		\begin{align}\label{thm2:eq2}
			|\Phi'(\lambda)-\Phi'(\lambda')|&\leq  \Bigl|\ln \frac{\lambda}{\lambda'}\Bigr|\Var(F(g)).
		\end{align}
	\end{enumerate}
\end{theorem}

The main feature of these three inequalities is involvement of $\Var(F(g)),$ in particular,  the first  is known as the one-sided strong sub-Gaussianity, see \cite{bobkov2022strongly}. 
If one assumes that $F$ is $L$-Lipschitz without convexity, it is standard to check, by using the Herbst argument and the log-Sobolev inequality, that the variance terms in the above inequalities are replaced by $L^2.$  Thus, Theorem \ref{thm2} improves the Lipschitz bounds when $F$ is convex.  
In the following subsections, we turn to the discussions of three major consequences of Theorems~\ref{convexity} and \ref{thm2}. 

\subsection{Application I: Deviation Inequalities}


Paouris and Valettas \cite{paouris2018gaussian} established a Gaussian inequality measuring the probability that $F(g)$ deviates away from its median $M$ from below. Later, Valettas (see Theorem 2.15 in \cite{valettas2019tightness} and related refinements in Note 2.17) sharpened their inequality with an optimal constant. More precisely, for any convex function $F$ with $\e |F(g)|<\infty$,
\begin{align}\label{gv1}
	\p\bigl(F(g)-M<-t\e (F(g)-M)_+\bigr)\leq \gamma\Bigl(\frac{t}{\sqrt{2\pi}}\Bigr),\,\,\forall t>0
\end{align}
and the equality is attained by affine functions $F(x)=u\cdot x+v$ for $u,v\in \mathbb{R}$ and $u\neq 0,$
where $$\gamma(x):=\frac{1}{\sqrt{2\pi}}\int_{x}^{\infty}e^{-\frac{r^2}{2}}dr.$$
By using the well-known bounds $\e (F(g)-M)_+\leq \sqrt{\mbox{Var}(F(g))}$ and $M\geq \e F(g)-\sqrt{\mbox{Var}(F(g))}$, they deduced two Gaussian inequalities from \eqref{gv1} on how $F(g)$ deviates away from $\e F(g)$ and $M$ respectively in terms of the standard deviation, which state that for convex $F$ with $\e F(g)^2<\infty,$  
\begin{align}\label{pvineq}
	\p\bigl(F(g)-\e F(g)\leq -t\sqrt{\Var(F(g))}\bigr)\leq \gamma\Bigl(\frac{t-1}{\sqrt{2\pi}}\Bigr),\,\,\forall t>1
\end{align}
and
\begin{align}\label{pvineq1}
	\p\bigl(F(g)-M\leq -t\sqrt{\mbox{Var}(F(g))}\bigr)\leq \gamma\Bigl(\frac{t}{\sqrt{2\pi}}\Bigr),\,\,\forall t>0. 
\end{align}
Note that $\lim_{x\to\infty}xe^{x^2/2}\gamma(x)=(2\pi)^{-1/2}$.  The leading quadratic functions in the exponents of these two upper bounds are approximately $-t^2/(4\pi).$
Our first application of Theorem~\ref{thm2} establishes an improvement of  \eqref{pvineq} and \eqref{pvineq1} for large $t$ that holds with a slightly better exponent coefficients.

\begin{theorem}\label{smalldev}
	Assume that $F$ is convex and $\e |F(g)|^2<\infty.$ We have that for any $t>0,$
	\begin{align*}
		\p\bigl(F(g)-\e F(g)\leq -t\sqrt{\Var(F(g))}\bigr)\leq e^{-\frac{t^2}{2}}
	\end{align*}
	and
	\begin{align*}
		\p\bigl(F(g)-M\leq -t\sqrt{{\Var}(F(g))}\bigr)\leq e^{-\frac{t^2}{2}}.
	\end{align*}
\end{theorem}

\begin{proof} To prove the first inequality, without loss of generality, assume that $\e F(g)=0$ and $\e F(g)^2=1.$
	From \eqref{thm2:eq1}, we have that $\e e^{\lambda F(g)}\leq e^{\lambda^2/2}$ and thus, for any $\lambda<0$ and $t\geq 0,$
	\begin{align*}
		\p(F(g)\leq -t)&\leq \e e^{\lambda (F(g)+t)}\leq e^{\lambda t+\lambda^2/2}.
	\end{align*}
	Minimizing over $\lambda<0$ implies that the minimum on the right-hand side must be attained by $\lambda =-t$ and this completes our proof of the first inequality. By using the observation $\mathbb E F(g) \geq M$ (see, e.g., \cite[Remark 2.4.2]{paouris2018gaussian}) and our first assertion, the second inequality holds immediately.	
\end{proof}

\begin{remark}
	\rm The well-known Gaussian  concentration inequality states that for any $L$-Lipschitz functions $F:\mathbb{R}^n\to \mathbb{R}$,
	\begin{align}
		\label{Gaussianineq}
		\p\bigl(|F(g)-\e F(g)|\geq tL\bigr)&\leq 2e^{-\frac{t^2}{2}},\,\,\forall t>0.
	\end{align}
	If one considers only the one-sided probability $\p(F(g)-\e F(g)\leq -tL)$, then it is upper bounded by $e^{-t^2/2}$ instead,  see, e.g., \cite[Theorem 5.6]{BLM}.
	When $F$ is convex, Theorem \ref{smalldev} improves this bound since $ \mbox{Var}(F(g))\leq L^2$. In particular, when $F(g)$ exhibits superconcentration, i.e., ${\mbox{Var}(F(g))}\ll {\e\|\nabla F(g)\|^2}$ (see \cite{superconc}), we have ${\mbox{Var}(F(g))}\ll L^2$ and the resulting bound of Theorem \ref{smalldev} largely improves the standard Gaussian bound \eqref{Gaussianineq}. 
\end{remark}

\begin{remark}\rm
	From the central limit theorem, both inequalities in Theorem \ref{smalldev} are asymptotically sharp in the sense that for any $F_n(x):=\sum_{i=1}^nf(x_i)$ for some convex function $f$ on $\mathbb{R}$ with $\e |f(z)|^2<\infty$ for $z\sim N(0,1)$, taking  $2t^{-2}\ln$ on the both sides of the two inequalities in Theorem \ref{smalldev} and sending $n\to\infty$ and then $t\to\infty,$ the limits all converge to one.
\end{remark}

\begin{remark}\rm
	Paouris-Valettas inequality \eqref{pvineq} implies that there exist $A>0$ and $0<a\leq 1/(4\pi)$ such that for any  convex $F$ with $\e F(g)^2<\infty$, 
	\begin{align}\label{PVineq}
		\p\bigl(F(g)-\e F(g)\leq -t\sqrt{\mbox{Var}(F(g))})
		\bigr)\leq Ae^{-at^2},\,\,\forall t\geq 0.
	\end{align}
	Following analogous arguments for establishing the equivalent conditions for sub-Gaussian random variables, see, e.g., \cite{vershynin2018high}, it can be checked that if $Y$ is a random variable with $\e Y=0$ and $\e Y^2=1$ and it satisfies $\p(Y\leq -t)\leq Ae^{-at^2}$ for all $t\geq 0,$ then there exists some $C=C(A,a)>1/2$ such that $\ln\e e^{\lambda Y}\leq C\lambda^2$ for all $\lambda<0.$ From this inequality, if we take $Y=(F(g)-\e F(g))/\sqrt{\mbox{Var}(F(g))},$ then  \eqref{add:eq1} and \eqref{thm2:eq1} are recoverable from \eqref{PVineq} with a suboptimal constant.
	
\end{remark}

\subsection{Application II: Dotsenko-Franz-M\'ezard Conjecture}\label{sec2}
The second application is concerned about the Sherrington-Kirkpatrick (SK) model.  It is a mean-field disordered spin system invented in order to understand some unusual magnetic behavior of certain alloys, see \cite{Sherrington1975}. For a given (inverse) temperature $\beta>0$ and an external field $h\in \mathbb{R},$ the Hamiltonian of the SK model is defined as
\begin{align*}
	H_N^x(\sigma)=\beta X_N^x(\sigma)+h\sum_{i=1}^N\sigma_i
\end{align*}
for $\sigma\in \{-1,1\}^N$ and $x=(x_{ij})_{1\leq i,j\leq N}\in \mathbb{R}^{N^2},$ where $X_N^x(\sigma):=N^{-1/2}\sum_{1\leq i,j\leq N}x_{ij}\sigma_i\sigma_j.$ Define the free energy as
$
F_N(x)=\ln Z_N(x)
$
for $$
Z_N(x):=\sum_{\sigma\in\{-1,1\}^N}e^{H_N^x(\sigma)}.
$$
One of the major questions in the SK model is to understand the limit of the normalized free energy $N^{-1}F_N(g)$, where $g$ is an $N^2$-dimensional standard normal vector.  First of all, it can be checked that $F_N$ is Lipschitz with $\|\nabla F_N(x)\|_2\leq \beta \|x\|_2\sqrt{N}$. With this, \eqref{Gaussianineq} readily implies that $N^{-1}F_N(g)$ is concentrated around its mean. Hence, to understand the limiting free energy, it suffices to study the limit of $N^{-1}\e F_N(g).$ To this end, theoretical physicists (see, e.g., \cite{mezard1987spin}) invented the so-called replica method by considering the annealed free energy of $\lambda$-replica,
\begin{align*}
	\Gamma_N(\lambda)=
	\frac{1}{\lambda N}\ln \e Z_N(g)^\lambda,\,\,\lambda\in \mathbb{R},
\end{align*}
where $\Gamma_N(0):=N^{-1}\e \ln Z_N(g)=N^{-1}\e F_N(g).$ 
Two key features of this function are the uniform Lipschitz propery in $\lambda$ and the convergence of $\Gamma_N$.

\begin{proposition}\label{Lipschitz}
	For any $\lambda,\lambda'\in \mathbb{R}$ and $N\geq 1,$ we have that
	\begin{align*}
		\bigl|\Gamma_N(\lambda)-\Gamma_N(\lambda')\bigr|
		\leq \frac{\beta^2|\lambda-\lambda'|}{2}.
	\end{align*}
	Furthermore, for any $\lambda\in \mathbb{R}$, $
	\Gamma(\lambda):=\lim_{N\to\infty}\Gamma_N(\lambda) $ exists.
\end{proposition}

\begin{proof}
	We follow a Herbst-type argument, see, e.g., \cite[Chapter 5]{BLM}. Let $\Phi_N=N\Gamma_N$.
	A direct differentiation and using the Jensen inequality imply that for any $\lambda\neq 0,$
	\begin{align*}\Phi_N'(\lambda)&=-\frac{1}{\lambda^2}\ln  \e Z_N(g)^{\lambda} +\frac{1}{\lambda}\frac{\e Z_N(g)^{\lambda} F_N(g) }{\e Z_N(g)^{\lambda} }\geq 0.
	\end{align*}
	On the other hand, by applying the log-Sobolev inequality to $f(x)=Z_N(x)^{\lambda/2}$, 
	\begin{align*}
		\e Z_N(g)^{\lambda}  \ln Z_N(g)^{\lambda} -\e Z_N(g)^{\lambda} \cdot \e\ln Z_N(g)^{\lambda} &\leq \frac{\lambda^2}{2}\e Z_N(g)^{\lambda} \|\nabla F_N(g)\|^2.
	\end{align*}
	Our proof is completed since the fact that $F_N$ is $\beta\sqrt{N}$-Lipschitz readily implies
	\begin{align*}
		0\leq\Phi_N'(\lambda)&\leq \frac{1}{2}\frac{\e \|\nabla F_N\|^2Z_N(g)^{\lambda} }{\e Z_N(g)^{\lambda} }\leq \frac{\beta^2N}{2}.
	\end{align*}
	The existence of $\Gamma$ follows by using Guerra-Toninelli's interpolation argument \cite{guerra2002thermodynamic}. Indeed, it can be shown that $(\lambda N\Gamma_N(\lambda))_{N\geq 1}$ is subadditive for $\lambda\geq 1$ and superadditive for $\lambda<1.$ As the argument is fairly standard, we omit the proof and refer the reader to check, for example, \cite{guerra2018emergence} for detail.
\end{proof}

When $\lambda$ is a positive integer,  it is well-known that $\lim_{N\to\infty}\Gamma_N(\lambda)$ can be explicitly computed as a variational formula parametrized by the so-called replica matrix (see, e.g., \cite[Theorem 1.13.1]{talagrand2010mean}). Now the idea of the replica method is to {\it assume} that this variational formula can be analytically continued to positive real numbers $\lambda$ near zero so that one can compute the limiting free energy through
\begin{align}\label{add:eq-1}
	\lim_{N\to\infty}	\frac{\e F_N(g)}{N}=\lim_{\lambda\downarrow 0}\Gamma(\lambda),
\end{align}
where the equality was due to Proposition \ref{Lipschitz}.  
In order to handle the limit on the right-hand side, G. Parisi \cite{parisi1980order,parisi1980sequence} proposed the so-called ultrametric ansatz in the computation of the variational formula for $\Gamma(\lambda),$ $\lambda>0$ and established the famous Parisi formula for the limiting free energy. Significantly different from physicists' replica method, the first proof of the Parisi formula was presented by Talagrand \cite{talagrand2006parisi} following Guerra's beautiful discovery of the replica symmetry breaking bound \cite{guerra2003broken}; an alternative proof later appeared and generalized to the mixed $p$-spin model in Panchenko \cite{panchenko2014parisi}.

Following the replica method, it is an intriguing question to understand the exact value and regularity of $\Gamma(\lambda)$  for any $\lambda\in \mathbb{R}.$
It turns out that 
for any $\lambda\geq 0,$ $\Gamma(\lambda)$ admits a Parisi-type formula, see Talagrand \cite{talagrand2007large} and Contucci-Mingione \cite{contucci2019multi}. When $\lambda<0$ and the external field vanishes, i.e., $h=0,$ it was conjectured by Dotsenko-Franz-M\'ezard (DFM) \cite{dotsenko1994partial} that $\Gamma(\lambda)=\Gamma(0)$ for all $\lambda<0.$ In  \cite{talagrand2007large}, Talagrand validated  this conjecture assuming the Parisi hypothesis, which, roughly speaking, states that with positive probability, there exists a sequence of subsets of spin configurations that charge deterministic weights and the spin configurations from distinct subsets are nearly orthogonal to each other, see \cite[Definition 10.1]{talagrand2007large}. This hypothesis was later justified in Jagannath \cite{jagannath2017approximate} relying on Parisi's ultrametricity \cite{panchenko2013parisi}.

Based on Theorem \ref{thm2}, we obtain an elementary proof for  the DFM conjecture \cite{dotsenko1994partial} with a quantitative bound. 
To see this, note that $F_N$ is convex since for any $x,y\in \mathbb{R}^{N^2}$ and $\alpha\in [0,1],$ by H\"older's inequality,
\begin{align*}
	F_N(\alpha x+(1-\alpha)y)&=\ln \sum_{\sigma\in\{-1,1\}^N}e^{\alpha H_N^x(\sigma)+(1-\alpha) H_N^y(\sigma)}\\
	&\leq \ln \Bigl(\sum_{\sigma\in \{-1,1\}^N}e^{H_N^x(\sigma)}\Bigr)^{\alpha}\Bigl(\sum_{\sigma\in \{-1,1\}^N}e^{H_N^y(\sigma)}\Bigr)^{1-\alpha}\\
	&=\alpha F_N(x)+(1-\alpha )F_N(y).
\end{align*}
When $h=0,$ it was shown in Chatterjee \cite{chatterjee2009disorder} that the free energy in the SK model is superconcentrated in the sense that its variance is of smaller order compared with the Poincar\'e inequality by deriving the bound $\Var(F_N(g))\leq CN/\ln  N$ for some universal constant $C$ depending only on $\beta.$ Plugging this bound and $F_N$  into \eqref{thm2:eq1}, we readily have

\begin{theorem}\label{DFMconjecture}
	Let $h=0$ and $\beta>0.$ There exists a positive constant $C$ depending only on $\beta$ such that for any $\lambda<0$ and $N\geq 1,$
	\begin{align*}
		|\Gamma_N(\lambda)-\Gamma_N(0)|\leq \frac{C|\lambda|}{\ln  N}.
	\end{align*}
\end{theorem}

We mention that another use of superconcentration inequality in spin glasses has also appeared in the study of the Thouless-Anderson-Palmer free energy landscape in the spherical mixed $p$-spin model, see Subag \cite{subag2024free}.

\subsection{Application III: Differentiability in Negative Replica}

As mentioned before, physicist's replica method for computing the limiting free energy $\lim_{N\to\infty}	N^{-1}\e F_N(g)$ heavily relies on assuming that the formula they obtained for $\Gamma(\lambda)$ for positive integers $\lambda$ can be analytically continued to any positive real $\lambda$ near the origin. Whether or not this is achievable remains unclear. However, when the external field is absent, Theorem \ref{DFMconjecture} implies that analytic continuation to negative $\lambda$ is not possible since $\Gamma(\lambda)$ is not a constant function near the origin. Our last result nevertheless shows that $\Gamma$ is differentiable for any negative $\lambda$ at any temperature and external field by using the third item of Theorem \ref{thm2}.

\begin{theorem}\label{thm3} For any $\beta$ and $h$, 
	$\Gamma$ is differentiable at any negative $\lambda$ and its derivative satisfies $$
	\Gamma'(\lambda)=\lim_{N\to\infty}\Gamma_N'(\lambda).$$
\end{theorem}

\begin{proof}
	Let $\lambda_0<0$ be fixed. From \eqref{thm2:eq2} and the convexity of $\Gamma_N,$ for any $\lambda\in (\lambda_0-\varepsilon,\lambda_0+\varepsilon),$ 
	\begin{align*}
		\Gamma_N'(\lambda_0)-\beta^2\Bigl|\ln \frac{\lambda_0-\varepsilon}{\lambda_0}\Bigr|&\leq \Gamma_N'(\lambda_0-\varepsilon)\\
		&\leq \frac{\Gamma_N(\lambda)-\Gamma_N(\lambda_0)}{\lambda-\lambda_0}\\
		&\leq \Gamma_N'(\lambda_0+\varepsilon)\leq \Gamma_N'(\lambda_0)+\beta^2\Bigl|\ln \frac{\lambda_0+\varepsilon}{\lambda_0}\Bigr|.
	\end{align*}
	Passing to the limit yields
	\begin{align*}
		\limsup_{N\to\infty}\Gamma_N'(\lambda_0)-\beta^2\Bigl|\ln \frac{\lambda_0-\varepsilon}{\lambda_0}\Bigr|\leq \frac{\Gamma(\lambda)-\Gamma(\lambda_0)}{\lambda-\lambda_0}\leq \limsup_{N\to\infty}\Gamma_N'(\lambda_0)+\beta^2\Bigl|\ln \frac{\lambda_0+\varepsilon}{\lambda_0}\Bigr|
	\end{align*}
	and the same inequality also holds if we take $\liminf_N$ instead.
	Hence, the differentiability of $\Gamma(\lambda)$ follows and $\Gamma(\lambda)=\lim_{N\to\infty}\Gamma_N'(\lambda).$
\end{proof}

\begin{remark}
	\rm Our approaches to Theorems \ref{DFMconjecture} and  \ref{thm3}  are not only limited to the SK model. Indeed, assume that $(F_N)_{N\geq 1}$ is a sequence of $C\sqrt{N}$-Lipschitz and convex functions defined on $\mathbb{R}^{k_N}$ for some constant $C>0$ independent of $N$ and some sequence of positive integers $\{k_N\}_{N\geq 1}.$ Define $$
	\Gamma_N(\lambda)=\frac{1}{\lambda N}\ln \e e^{\lambda F_N(g)},\,\,\lambda\in \mathbb{R}$$
	for an $k_N$-dimensional standard normal $g$. If one assumes that $F_N(g)$ is superconcentrated, i.e., $\lim_{N\to\infty}N\Var(F_N(g))=0,$ then DFM conjecture holds, i.e., $\limsup_{N\to\infty}|\Gamma_N(\lambda)-\Gamma_N(0)|=0$ for all $\lambda<0$. Also, if $\Gamma(\lambda)=\lim_{N\to \infty}\Gamma_N(\lambda)$ converges for $\lambda\in \mathbb{R},$ then $\Gamma$ is always differentiable for negative $\lambda.$ Examples of $F_N$ that exhibit superconcentration beyond the SK model include the free energies for the mixed even $p$-spin model and the Gaussian directed polymer without external field, see \cite{chatterjee2008chaos,chatterjee2009disorder}.
	Also, one can employ Guerra-Toninelli's argument, see \cite{guerra2018emergence,guerra2002thermodynamic}, to show that the sequence $(\lambda N\Gamma_N(\lambda))_{N\geq 1}$ is subadditivity for $\lambda\geq 1$ and superadditivity for $\lambda<1$ in the mixed even $p$-spin model. As for the Gaussian directed polymer, the same additivity remains valid, where the proof can be easily established using the translation invariance of the random environment, see \cite{comets2017directed}. As a consequence, the corresponding $\Gamma_N$ in these models converges at any negative $\lambda.$
\end{remark}

\subsection{Discussion}

Following Theorem \ref{convexity}, it is natural to ask if they hold for more general random vectors.
One simple extension is to consider $X=\rho(g)$ for some $\phi:\mathbb{R}^l\to \mathbb{R}^n$ and $g$ an $l$-dimensional standard normal. If $F:\mathbb{R}^n\to \mathbb{R}$ such that $F_0(x):=F(\rho(x))$ is  convex on $\mathbb{R}^l,$ then the conclusions of Theorems \ref{convexity}, \ref{thm2}, and \ref{smalldev} as well as \eqref{gv1}, \eqref{pvineq}, and \eqref{pvineq1} hold for $F(X)$. As an example observed earlier in \cite{paouris2018gaussian,paouris2019variance,valettas2019tightness}, for any $k\geq 1,$ if $l=kn$ and $\rho(x)=(\sum_{j=k(i-1)+1}^{ki}x_{j}^2)_{1\leq i\leq n},$ then $\rho(g)$ is a vector of $n$ i.i.d. $\chi^2(k)$ random variables. If $F$ is assume to be convex and 1-conditional, i.e., $F(y_1,\ldots,y_n)=F(|y_1|,\ldots,|y_n|)$,  then $F_0$ is convex.

Recently, Paouris-Valettas \cite{paouris2019variance} established an analogous bound of \eqref{gv1} for any log-concave probability measures on $\mathbb{R}^n$ with a sub-exponential tail and studied an example that asymptotically attains their sub-exponential bound. Their result indicates that to obtain the sub-Gaussian tail bound for the small deviation, more restrictions on $F$ in addition to being convex seem to be necessary.

We close this section by mentioning that establishing convexity/concavity for analogous logarithmic moment generating functions with respect to other measures is also possible by adding an adjustment term. For example, Borell \cite{borell1973complements} showed that for any nonnegative, bounded, and convex function $f$ on a convex body\footnote{A convex body is a nonempty, open, bounded, and convex subset of $\mathbb{R}^n.$} $\Omega\subseteq\mathbb{R}^n$, $0\leq \lambda\mapsto\ln  (n+\lambda)+\lambda^{-1}\ln \int_{\Omega} f(x)^\lambda dx$ is convex. On the other hand, Bobkov \cite{bobkov2003spectral} obtained that $0< \lambda\mapsto -\lambda\ln \lambda+\ln \e X^\lambda$ is concave if $X$ is a nonnegative random variable with a log-concave distribution.


\section{Proofs of Theorem \ref{convexity} and \ref{thm2}}\label{proof}

Our proof of Theorem \ref{convexity} is based on the stochastic optimal control representation derived by Bou\'e-Dupuis \cite{boue1998variational}. For completeness, we provide an elementary argument to derive their representation. Incidentally, similar approaches to establishing some convexities or geometric inequalities by analogous representations have appeared earlier in the literature, see, e.g., \cite{auffinger2015parisi,borell2000diffusion,lehec2014short}.

Let $W=(W(s))_{0\leq s\leq 1}$ be an $n$-dimensional standard Brownian motion on $[0,1]$ and $\mathcal{D}$ be the collection of all $n$-dimensional progressively measurable processes $u=(u(s))_{0\leq s\leq 1}$ with respect to the filtration generated by $W$ and $\e\int_0^1\|u(s)\|^2ds<\infty$.

\begin{lemma}[Bou\'e-Dupuis representation]\label{BD}
	Assume that $F$ is a Lipschitz function on $\mathbb{R}^n$. We have that
	\begin{align}
		\frac{1}{\lambda}\ln\e e^{\lambda F(g)}&=\left\{
		\begin{array}{ll}	\label{rep1}
			\sup_{u\in \mathcal{D}}\e \Bigl[F\Bigl(\lambda\int_0^1 u(s)ds+W(1)\Bigr)-\frac{\lambda}{2}\int_0^1\|u(s)\|^2ds\Bigr],&\mbox{if $\lambda\geq 0$},\\
			\\
			\inf_{u\in \mathcal{D}}\e\Bigl[F\Bigl(\lambda\int_0^1 u(s)ds+W(1)\Bigr)-\frac{\lambda}{2}\int_0^1\|u(s)\|^2ds\Bigr],&\mbox{if $\lambda<0$}.
		\end{array}
		\right.
	\end{align}
	
\end{lemma}
\begin{proof}
	If $\lambda=0,$ our result trivially holds. We only establish the first equation for $\lambda>0$ since it clearly recovers the second one. For any $\lambda>0,$ $x\in \mathbb{R}^n$, and $s\in [0,1],$ define 
	\begin{align*}
		Q (s,x)=\frac{1}{\lambda}\ln  \e e^{\lambda F(x+\sqrt{1-s}g)}=\frac{1}{\lambda}\ln \int e^{\lambda F(t)}p_s(t,x)dt
	\end{align*}
	for $p_{s}(t,x)=(2\pi (1-s))^{-n/2}e^{-\|t-x\|^2/2(1-s)}.$ Note that this function is well-defined since $F$ is Lipschitz. From the dominated convergence theorem, it can be computed directly that
	\begin{align*}
		\partial_sQ (s,x)&=\frac{1}{2\lambda \e e^{\lambda F(x+\sqrt{1-s}g)}}\Bigl(\frac{n}{1-s}\int e^{\lambda F(t)}p_s(t,x)dt\\
		&\qquad\qquad-\frac{1}{(1-s)^2}\int e^{\lambda F(t)}\|t-x\|^2p_s(t,x)dt\Bigr)
	\end{align*}
	and
	\begin{align*}
		\nabla Q (s,x)&=\frac{1}{\lambda\e e^{\lambda F(x+\sqrt{1-s}g)}}\int e^{\lambda F(t)}\frac{t-x}{1-s}p_s(t,x)dt,\\
		\Delta Q (s,x)&=\frac{1}{\lambda \e e^{\lambda F(x+\sqrt{1-s}g)}}\Bigl(\frac{-1}{\e e^{\lambda F(x+\sqrt{1-s}g)}}\sum_{i=1}^n\Bigl(\int e^{\lambda F(t)}\frac{(t_i-x_i)}{1-s}p_s(t,x)dt\Bigr)^2\\
		&\qquad\qquad -\frac{n}{1-s}\int e^{\lambda F(t)}p_s(t,x)dt+\frac{1}{(1-s)^2}\int e^{\lambda F(t)}\|t-x\|^2p_s(t,x)dt\Bigr).
	\end{align*}
	From these, $Q $ satisfies
	\begin{align*}
		\partial_sQ (s,x)&=-\frac{1}{2}\bigl(\Delta Q (s,x)+\lambda \bigl\|\nabla Q (s,x)\bigr\|^2\bigr)
	\end{align*}
	with boundary condition $Q (1,x)=F(x).$ Now, for any $u\in \mathcal{D}$, consider the process
	$$
	Y(s)=x+\lambda \int_0^su(t)dt+W(s).
	$$
	From It\'{o}'s formula,
	\begin{align*}
		dQ (s,Y(s))
		&=\partial_sQ (s,Y(s))ds+\nabla Q (s,Y(s))\cdot dY(s)+\frac{1}{2}\Delta Q (s,Y(s))ds\\
		&=-\frac{1}{2}\bigl(\Delta Q (s,Y(s))+\lambda \bigl\|\nabla Q (s,Y(s))\bigr\|^2\bigr)ds\\
		&\quad+\lambda \nabla Q (s,Y(s))\cdot u(s)ds+\nabla Q (s,Y(s))\cdot dW(s)+\frac{1}{2}\Delta Q (s,Y(s))ds\\
		&=-\frac{\lambda}{2}\bigl\|\nabla Q (s,Y(s))\bigr\|^2ds+\lambda \nabla Q (s,Y(s))\cdot u(s)ds+\nabla Q (s,Y(s))\cdot dW(s)\\
		&=-\frac{\lambda}{2}\bigl\|\nabla Q (s,Y(s))-u(s)\bigr\|^2ds+\frac{\lambda}{2}\|u(s)\|^2ds+\nabla Q (s,Y(s))\cdot dW(s).
	\end{align*}
	Therefore,
	\begin{align*}
		Q (1,Y(1))-Q (0,Y(0))&=-\frac{\lambda}{2}\int_0^1\bigl\|\nabla Q (s,Y(s))-u(s)\bigr\|^2ds\\
		&+\frac{\lambda}{2}\int_0^1\|u(s)\|^2ds+\int_0^1\nabla Q (s,Y(s))\cdot dW(s).
	\end{align*}
	This implies that
	\begin{align*}
		Q (0,x)&=F\Bigl(x+\lambda \int_0^1 u(s)ds+W(1)\Bigr)-\frac{\lambda}{2}\int_0^1\|u(s)\|^2ds\\
		&+\frac{\lambda}{2}\int_0^1\bigl\|\nabla Q (s,Y(s))-u(s)\bigr\|^2ds-\int_0^1\nabla Q (s,Y(s))\cdot dW(s)
	\end{align*}
	and it follows that since $\lambda \geq 0,$
	\begin{align*}
		Q (0,x)&\geq \sup_{u}\e\Bigl[F\Bigl(x+\lambda \int_0^1 u(s)ds+W(1)\Bigr)-\frac{\lambda}{2}\int_0^1\|u(s)\|^2ds\Bigr].
	\end{align*}
	Here, the equality holds by noting that in the previous display, the square term vanishes  if $u(s)=\nabla Q (s,Y(s)),$ where $Y$ is the strong solution to the stochastic differential equation $dY(s)=\lambda \nabla Q (s,Y(s))ds+dW(s)$  with initial condition $Y(0)=x.$ Finally, letting $x=0$, we establish the first line of \eqref{rep1}.
\end{proof}

\begin{lemma}\label{lem0}
	Assume that $F$ is a convex function on $\mathbb{R}^n$. There exists a sequence of  convex and Lipschitz functions $(F_r)_{r\geq 1}$ on $\mathbb{R}^n$ such that $F_r\leq F_{r+1}$, $F_r(x)=F(x)$ for any $\|x\|\leq r$, and $F_r(x)\leq F(x)$ for any $\|x\|\geq r.$
\end{lemma}

\begin{proof}
	For $a\in \mathbb{R}^n,$ let $S(a)$ be the collection of  all $m\in \mathbb{R}^n$ such that $m\cdot (x-a)+F(a)\leq F(x)$ for all $x\in \mathbb{R}^n.$ In other words, $S(a)$ is the collection of all supporting hyperplanes at $a.$ For any $r\geq 1,$ define
	\begin{align*}
		F_r(x):=\sup\bigl\{m\cdot(x-a)+F(a):m\in S(a)\,\,\mbox{for some $a$ satisfying $\|a\|\leq r$}\bigr\},\,\,x\in \mathbb{R}^n.
	\end{align*}
	It is easy to see that $F_r\leq F_{r+1}.$
	For any $m\in S(a)$ for some $\|a\|\leq r,$ since $m\cdot(x-a)+F(a)$ is linear in $x$, $F_r$ is convex; also, since $m\cdot(x-a)+F(a)\leq F(x)$ for any $x\in \mathbb{R}^n$, we have that $F_r(x)\leq F(x)$ for any $x\in\mathbb{R}$ and by the convexity of $F,$ $F_r(x)=F(x)$ if $\|x\|\leq r.$ Lastly, if $m\in S(a)$, then for any $1\leq i\leq n,$ plugging $x=a+\mbox{sign}(m_i)e_i$ into $m\cdot (x-a)+F(a)\leq F(x)$ yields $$
	|m_i|\leq F(a+\mbox{sign}(m_i)e_i)-F(a)\leq \max\bigl(|F(a+e_i)-F(a)|,|F(a-e_i)-F(a)|\bigr),
	$$
	where $\{e_1,\ldots,e_n\}$ is the standard basis of $\mathbb{R}^n.$ Letting $$
	M_r=\sup\Bigl\{\Bigl(\sum_{i=1}^n\max\bigl(|F(a+e_i)-F(a)|,|F(a-e_i)-F(a)|\bigr)^2\Bigr)^{1/2}:\|a\|\leq r\Bigr\},
	$$
	we see that $\|m\|\leq M_r$ for all $m\in S(a)$ for $\|a\|\leq r.$ Therefore, $F_r$ is $M_r$-Lipschitz. Our proof is completed.
	
\end{proof}

\begin{proof}[\bf Proof of Theorem \ref{convexity}]
	
	We first establish our assertion with an additional assumption that $F$ is Lipschitz, which allows us to use the representation in Lemma  \ref{BD}.
	Now as the first line of \eqref{rep1} is a maximization problem and $F(x)$ is convex, we see immediately that $\Phi(\lambda)$ is a convex function for $\lambda\geq 0$. The convexity of $\Phi(\lambda)$ for $\lambda<0$ is based on the second representation of
	\eqref{rep1}.
	For any $\lambda_0,\lambda_1<0$, $\alpha\in (0,1)$, and $u_0,u_1$, since $F$ is convex, we have
	\begin{align*}
		&\alpha F\Bigl(\lambda_0\int_0^1 u_0(s)ds+W(1)\Bigr)+(1-\alpha)F\Bigl(\lambda_1\int_0^1 u_1(s)ds+W(1)\Bigr)\\
		&\geq F\Bigl(\int_0^1(\alpha \lambda_0u_0(s)+(1-\alpha)\lambda_1u_1(s))ds+W(1)\Bigr)\\
		&=F\Bigl(\lambda_\alpha\int_0^1u_\alpha(s)ds+W(1)\Bigr),
	\end{align*}
	where 
	\begin{align*}
		\lambda_\alpha&:=\alpha \lambda_0+(1-\alpha)\lambda_1,\\
		u_\alpha&:=\frac{\alpha \lambda_0u_0+(1-\alpha)\lambda_1u_1}{\alpha \lambda_0+(1-\alpha)\lambda_1}.
	\end{align*}
	On the other hand, 
	\begin{align*}
		&-\alpha\lambda_0\int_0^1\e\|u_0(s)\|^2ds-(1-\alpha)\lambda_1\int_0^1\e\|u_1(s)\|^2ds\\
		&=\e \int_0^1\bigl(-\alpha \lambda_0\|u_0(s)\|^2-(1-\alpha) \lambda_1\|u_1(s)\|^2\bigr)ds.
	\end{align*}
	If we let $-\lambda_0=m_0$ and $-\lambda_1=m_1>0$, then using the convexity of $\|\cdot\|^2$ gives
	\begin{align*}
		-\alpha \lambda_0\|u_0\|^2-(1-\alpha) \lambda_1\|u_1\|^2&=\alpha m_0\|u_0\|^2+(1-\alpha)m_1\|u_1\|^2\\
		&=(\alpha m_0+(1-\alpha)m_1)\frac{\alpha m_0\|u_0\|^2+(1-\alpha)m_1\|u_1\|^2}{\alpha m_0+(1-\alpha)m_1}\\
		&\geq (\alpha m_0+(1-\alpha)m_1)\Bigl\|\frac{\alpha m_0u_0+(1-\alpha) m_1u_1}{\alpha m_0+(1-\alpha)m_1}\Bigr\|^2\\
		&=-\lambda_\alpha \|u_\alpha\|^2.
	\end{align*}
	Therefore, from the second line of \eqref{rep1},
	\begin{align*}
		&\alpha \Bigl(\e F\Bigl(\lambda_0\int_0^1 u_0ds+W(1)\Bigr)-\frac{\lambda_0}{2}\int_0^1\e\|u_0\|^2ds\Bigr)\\
		&+(1-\alpha)\Bigl(\e F\Bigl(\lambda_1\int_0^1 u_1ds+W(1)\Bigr)-\frac{\lambda_1}{2}\int_0^1\e\|u_1\|^2ds\Bigr)\\
		&\geq \e F\Bigl(\lambda_\alpha\int_0^1 u_\alpha ds+W(1)\Bigr)-\frac{\lambda_\alpha}{2}\int_0^1\e\|u_\alpha\|^2ds\geq \Phi(\lambda_\alpha).
	\end{align*}
	Since this holds for any $u_0$ and $u_1,$ taking infimum separatesly for $u_0$ and $u_1$ yields our assertion by using again the second representation \eqref{rep1}. Thus, we have completed our proof when $F$ is Lipschitz.
	
	Next, we consider the general case. From Lemma \ref{lem0}, there exists a sequence of convex and Lipschitz functions $F_r$ that satisfies $F_r\uparrow F$ as $r\to\infty.$ Based on our first part of the proof, it suffices to show that $\Phi_r(\lambda):=\lambda^{-1}\ln \e e^{\lambda F_r(g)}$ converges to $\Phi(\lambda)$ for any $\lambda\in \mathbb{R}.$ Note that $e^{\lambda F_r}\leq e^{\lambda F}$ when $\lambda\geq 0$ and $e^{\lambda F_r}\leq e^{\lambda F_1}$ when $\lambda<0.$ Since $\e e^{\lambda F}$ and $\e e^{\lambda F_1}$ are finite, the dominated convergence theorem implies the desired convergence.
\end{proof}

\begin{proof}[\bf Proof of Theorem \ref{thm2}]
	We first establish our result under an additional assumption that $F$ is Lipschitz, which will allow us to compute derivatives in the following argument. Write $$\Phi(\lambda)=\sqrt{\Var(F(g))}\Psi(\lambda\sqrt{\Var(F(g))})+\e F(g),$$
	where $$\Psi(r):=\frac{1}{r}\ln \e \exp\Bigl( \frac{r(F(g)-\e F(g))}{\sqrt{\Var(F(g))}}\Bigr).$$ 
	From this, our assertion is equivalent to the following inequalities:
	\begin{enumerate}
		\item[(i)] For all $\lambda\leq 0,$
		$$\lambda\Psi(\lambda)\leq \frac{\lambda^2}{2}.$$
		\item[(ii)] For all $\lambda \leq 0,$
		\begin{align*}
			|\Psi(\lambda)|\leq \frac{|\lambda|}{2}.
		\end{align*}
		\item[(iii)] For all $\lambda,\lambda'<0,$
		\begin{align*}
			|\Psi'(\lambda)-\Psi'(\lambda')|\leq \Bigl|\ln \frac{\lambda}{\lambda'}\Bigr|.
		\end{align*}
	\end{enumerate}
	From these, without loss of generality, we assume that $\e F(g)=0$ and $\e F(g)^2=1$ and we aim to show that the inequalities above hold for $\Phi=\Psi$. To establish (i), a direct differentiation gives that
	\begin{align}\label{add:eq2}
		\Phi'(\lambda)&=-\frac{1}{\lambda^2}\ln\e e^{\lambda F(g)}+\frac{\e F(g)e^{\lambda F(g)}}{\lambda\e e^{\lambda F(g)}}.
	\end{align} 
	From the convexity of $\Phi$ and L'Hopital's rule,
	\begin{align}
		\nonumber		\Phi'(\lambda)	&\leq \lim_{\lambda\to 0^-}\Phi'(\lambda)\\
		\nonumber			&=\lim_{\lambda\to 0^-}\frac{1}{\lambda^2}\Bigl(\frac{\lambda\e F(g)e^{\lambda F(g)}}{\e e^{\lambda F(g)}}-\ln\e e^{\lambda F(g)}\Bigr)\\
		\nonumber			&=\lim_{\lambda\to 0^-}\frac{1}{2\lambda}\Bigl(\lambda\Bigl(\frac{\e F(g)^2e^{\lambda F(g)}}{\e e^{\lambda F(g)}}-\Bigl(\frac{\e F(g)e^{\lambda F(g)}}{\e e^{\lambda F(g)}}\Bigr)^2\Bigr)+\frac{\e F(g)e^{\lambda F(g)}}{\e e^{\lambda F(g)}}-\frac{\e F(g)e^{\lambda F(g)}}{\e e^{\lambda F(g)}}\Bigr)\\
		\nonumber			&=\lim_{\lambda\to 0^-}\frac{1}{2}\Bigl(\frac{\e F(g)^2e^{\lambda F(g)}}{\e e^{\lambda F(g)}}-\Bigl(\frac{\e F(g)e^{\lambda F(g)}}{\e e^{\lambda F(g)}}\Bigr)^2\Bigr)\\
		\label{eq2}		&=\frac{1}{2}.
	\end{align}
	Thus, $$\ln \e e^{\lambda F(g)}=\lambda\Phi(\lambda)=\lambda\Bigl(\Phi(0)+\int_0^\lambda\Phi'(t)dt\Bigr)=-\lambda\int_{\lambda}^0\Phi'(t)dt\leq \frac{\lambda^2}{2}.$$
	For (ii), using the convexity of $\Phi$ and \eqref{eq2} leads to
	$$
	\frac{\Phi(\lambda)}{\lambda}=\frac{\Phi(0)-\Phi(\lambda)}{0-\lambda}\leq \lim_{t\to 0^-}\Phi'(t)=\frac{1}{2}.
	$$
	On the other hand, note that from the Jensen inequality and $\e F(g)=0,$ we have $\Phi(\lambda)\leq 0.$
	Therefore, $\Phi(\lambda)/\lambda\geq 0$ and consequently,
	$
	|\Phi(\lambda)|\leq 2^{-1}|\lambda|,
	$
	completing the proof of (ii). 
	For (iii),
	a direct differentiation gives
	\begin{align*}
		\Phi''(\lambda)	&=-\frac{2}{\lambda^3}\Bigl(\frac{\e e^{\lambda F(g)}\ln e^{\lambda F(g)}}{\e e^{\lambda F(g)}}-\ln \e e^{\lambda F(g)}\Bigr)\\
		&\quad+\frac{1}{\lambda^3}\Bigl(\frac{\e e^{\lambda F(g)} (\ln e^{\lambda F(g)})^2}{\e e^{\lambda F(g)}}-\Bigl(\frac{\e e^{\lambda F(g)} \ln e^{\lambda F(g)} }{\e e^{\lambda F(g)}}\Bigr)^2\Bigr).
	\end{align*}
	Note that from Jensen's inequality, the first term is positive and the second term is negative. By dropping the latter and using the convexity of $\Phi$, \eqref{add:eq2}, and \eqref{eq2}, we have that for any $\lambda<0,$
	\begin{align*}
		0\leq \Phi''(\lambda)\leq \frac{2}{|\lambda|^3}\Bigl(\frac{\e e^{\lambda F(g)}\ln e^{\lambda F(g)}}{\e e^{\lambda F(g)}}-\ln \e e^{\lambda F(g)}\Bigr)=\frac{2\Phi'(\lambda)}{|\lambda|}\leq \frac{1}{|\lambda|}.
	\end{align*}
	Therefore, for any $\lambda<\lambda'<0,$
	\begin{align*}
		|\Phi'(\lambda)-\Phi'(\lambda')|&=\int_{\lambda}^{\lambda'}\Phi''(t)dt\leq \int_{\lambda}^{\lambda'}\frac{1}{t}dt= \Bigl|\ln\frac{\lambda'}{\lambda}\Bigr|.
	\end{align*}
	
	For the general case, i.e., dropping the Lipschitz assumption, it also suffices to establish (i), (ii), and (iii) for $\Phi=\Psi.$ Now in the same notation at the end of the proof of Theorem \ref{convexity}, (i), (ii), and (iii) hold for $\Phi_r$ and from the same approximation argument therein, $\Phi$ satisfies (i) and (ii). As for (iii), note that $|F_r|e^{\lambda F_r}\leq \max(|F_1|,|F|)e^{\lambda F_1}$ and the Cauchy-Schwarz inequality implies $$
	\e \max(|F_1|,|F|)e^{\lambda F_1}\leq ((\e |F_1|^2)^{1/2}+\e|F|^2)^{1/2}(\e e^{2\lambda F_1})^{1/2}<\infty.$$
	From $F_re^{\lambda F_r}\to Fe^{\lambda F}$ and the dominated convergence theorem, we see that $\Phi_r'(\lambda)\to \Phi'(\lambda)$ and our proof is finished since $\Phi_r'$ satisfies (iii). 
\end{proof}

{\noindent \bf Acknowledgements}
This work was highly motivated by Dmitry Panchenko, who shared the validity of \eqref{add:eq-1} with an argument via the Gaussian concentration inequality \eqref{Gaussianineq} at the first place and by Francesco Guerra for the inspiring survey \cite{guerra2018emergence}. He thanks Sergey Bobkov for bringing \cite{bobkov2003spectral,borell1973complements} to his attention and Grigoris Paouris and Petros Valettas for many valuable comments, in particular, pointing out the sharp result \eqref{gv1} in the literature. 
This work is  partly supported by NSF grants, DMS-1752184 and DMS-2246715, and Simons Foundation grant 1027727.

\bibliographystyle{plain}
{\footnotesize\bibliography{ref}}

\end{document}